\documentclass[oneside,british,english,american]{amsart}
\usepackage[T1]{fontenc}
\usepackage[latin9]{inputenc}
\usepackage{color}
\usepackage{babel}
\usepackage{float}
\usepackage{mathtools}
\usepackage{amsthm}
\usepackage{amssymb}
\usepackage[unicode=true,pdfusetitle,
 bookmarks=true,bookmarksnumbered=false,bookmarksopen=false,
 breaklinks=false,pdfborder={0 0 0},backref=false,colorlinks=false]
 {hyperref}

\makeatletter

\providecommand{\tabularnewline}{\\}

\numberwithin{equation}{section}
\numberwithin{figure}{section}
 \theoremstyle{definition}
 \newtheorem*{defn*}{\protect\definitionname}
\theoremstyle{plain}
\newtheorem{thm}{\protect\theoremname}
  \theoremstyle{plain}
  \newtheorem{lem}[thm]{\protect\lemmaname}
  \theoremstyle{remark}
  \newtheorem{rem}[thm]{\protect\remarkname}
  \theoremstyle{plain}
  \newtheorem{cor}[thm]{\protect\corollaryname}

\makeatother

  \addto\captionsamerican{\renewcommand{\corollaryname}{Corollary}}
  \addto\captionsamerican{\renewcommand{\definitionname}{Definition}}
  \addto\captionsamerican{\renewcommand{\lemmaname}{Lemma}}
  \addto\captionsamerican{\renewcommand{\remarkname}{Remark}}
  \addto\captionsamerican{\renewcommand{\theoremname}{Theorem}}
  \addto\captionsbritish{\renewcommand{\corollaryname}{Corollary}}
  \addto\captionsbritish{\renewcommand{\definitionname}{Definition}}
  \addto\captionsbritish{\renewcommand{\lemmaname}{Lemma}}
  \addto\captionsbritish{\renewcommand{\remarkname}{Remark}}
  \addto\captionsbritish{\renewcommand{\theoremname}{Theorem}}
  \addto\captionsenglish{\renewcommand{\corollaryname}{Corollary}}
  \addto\captionsenglish{\renewcommand{\definitionname}{Definition}}
  \addto\captionsenglish{\renewcommand{\lemmaname}{Lemma}}
  \addto\captionsenglish{\renewcommand{\remarkname}{Remark}}
  \addto\captionsenglish{\renewcommand{\theoremname}{Theorem}}
  \providecommand{\corollaryname}{Corollary}
  \providecommand{\definitionname}{Definition}
  \providecommand{\lemmaname}{Lemma}
  \providecommand{\remarkname}{Remark}
\providecommand{\theoremname}{Theorem}

\begin{document}
\selectlanguage{british}%
\global\long\def\mod{\,\mathrm{mod\,\,}}
\global\long\def\li{\mathrm{li}}

\selectlanguage{american}%

\title{on the regularity of primes in arithmetic progressions}
\begin{abstract}
We prove that for a positive integer $k$ the primes in certain kinds
of intervals can not distribute too ``uniformly'' among the reduced
residue classes modulo $k$. Hereby, we prove a generalization of
a conjecture of Recaman and establish our results in a much more general
situation, in particular for prime ideals in number fields.
\end{abstract}

\author{Christian Elsholtz\and{}Niclas Technau\and{}Robert Tichy}

\thanks{The second author was supported by the Austrian Science Fund (FWF):
W1230 Doctoral Program ``Discrete Mathematics'', and the third author
was supported by the Special Research Program (SFB): SFB F 55 ``Quasi-Monte
Carlo Methods: Theory and Applications\textquotedbl{}.}

\address{Institut für Analysis und Computational Number Theory\\
Technische Universität Graz\\
Steyrergasse 30\\
A-8010\\
elsholtz@tugraz.at}

\address{Institut für Analysis und Computational Number Theory\\
Technische Universität Graz\\
Steyrergasse 30\\
A-8010\\
technau@math.tugraz.at}

\address{Institut für Analysis und Computational Number Theory\\
Technische Universität Graz\\
Steyrergasse 30\\
A-8010\\
tichy@tugraz.at}

\maketitle

\section{Introduction and Main Result}

Let $\omega\left(k\right)$ be the number of distinct prime factors
of an integer $k$, and let $\varphi$ denote Euler's totient function.
We say that $k$ is a P-integer if the first $\varphi\left(k\right)$
primes which do not divide $k$ form a complete residue system modulo
$k$. In 1978 Recaman \cite{Recam1978} conjectured that there are
only finitely many prime $P$-integers. In 1980 Pomerance \cite{Pomerance1980}
proved this by showing that there are in fact only finitely many $P$-integers,
and conjectured moreover that every $P$-integer does not exceed $30$.
This was proved in special cases by Hajdu, Saradha, and Tijdeman \cite{L.Hajdu2011,L.Hajdu2012,Saradha2011}.
In fact, they proved in \cite{L.Hajdu2012} the conjecture of Pomerance
under the assumption of the Riemann Hypothesis. Eventually, in a recent
paper of Yang and Togbé \cite{S.Yang2014} the conjecture was proven
unconditionally.

However, one can rephrase the definition of $P$-integers as follows:
Let, without further mentioning, $p$ denote a prime, $\mathbb{P}$
the set of primes, and $p_{n}$ the $n$-th smallest prime. Then $k$
is a $P$-integer if the block $p_{1},p_{2},\ldots,p_{\varphi\left(k\right)+\omega\left(k\right)}$
of the first $\varphi\left(k\right)+\omega\left(k\right)$ primes,
lying in the closed interval $\bigl[p_{1},p_{\varphi\left(k\right)+\omega\left(k\right)}\bigr]$,
has precisely one element in each reduced residue class modulo $k$,
with the exception of $\omega\left(k\right)$ primes (which lie in
distinct, non-invertible residue classes). By viewing $P$-integers
as instances of such distribution phenomena, there is an obvious and
far more general notion for this.
\begin{defn*}
\label{def: Pinteger}Let $\alpha,\beta,\gamma,\iota>0$ denote integers,
and $G=\left(G,\cdot\right)$ an arithmetical semi-group with norm
$\left|\cdot\right|$, in the sense of Knopfmacher \cite[p. 11]{Knopfmacher1990},
which takes only values in the positive integers. Consider for $k\in G$
the equivalence relation $a\sim b\,:\Leftrightarrow\left|a\right|\equiv\left|b\right|\mod\left|k\right|$
on $G$ and let $M$ denote the primes in $G$ with norm in the interval
$\left[\alpha,\beta\right]$. Then we say $k\in G$ is a $P\left(\alpha,\beta,\gamma,\iota\right)$-integer
if $M$ has in each equivalence class corresponding to an invertible
residue class modulo $\left|k\right|$ at least $\gamma$ elements,
and the remaining $\iota$ primes distribute in some arbitrary equivalence
classes such that $\left|M\right|=\gamma\varphi\left(k\right)+\iota$.
(For ease of exposition we shall simply speak of $P^{*}$-integers
if no confusion can arise.)
\end{defn*}
A natural question is to estimate for a given $k\in G$ the smallest
values of $\alpha,\beta$ such that $k$ is for the first time a $P^{*}$-integer.
Let us simplify this question by considering the semi-group $G=\mathbb{N}$
of the natural numbers, endowed with its canonical norm, and by asking
the following question: Fix $\alpha=2$ and estimate for a given $k$
the smallest integer $\beta=\beta\left(k\right)$ such that $k$ is
the first time a $P\left(2,\beta,1,\iota\right)$-integer for some
$\iota$. This problem is nothing but estimating Linnik's constant
which is widely open. Yet, the following probabilistic considerations
suggest that $\beta$ should be in $O(k\log^{2}k)$:

We start by estimating the probability $P(X)$ for a random set of
$f\left(k\right)\geq\varphi\left(k\right)$ primes to \emph{not} cover
all of the $\varphi\left(k\right)$ reduced residue classes with at
least one prime. We assume that a given prime $p$ is in an invertible
residue class modulo $k$; for otherwise $p$ divides $k$ which can
happen at most\footnote{By $\ll$ we denote the usual Vinogradov-symbol. The implied constant
is absolute unless we specify a dependency of some variable by an
appropriate subscript. We shall use the Landau-symbols in the same
way.} $\omega\left(k\right)\ll\log k$ times and picking such a prime out
of $f\left(k\right)$ arbitrary primes almost never happens. In view
of Dirichlet's Theorem on arithmetic progressions, we assume moreover
that a prime $p$ has about probability $\frac{1}{\varphi\left(k\right)}$
to be in a specific invertible residue class modulo $k$. Let $X_{r}$
denote the event that in the invertible residue class $r$ modulo
$k$ none of the $f\left(k\right)$ primes occurs. Then, writing $f\left(k\right)=C\left(k\right)\varphi\left(k\right)\log k$,
(say), the probability $P(X_{r})$ of $X_{r}$ is 
\[
\left(1-\frac{1}{\varphi\left(k\right)}\right)^{f\left(k\right)}\approx\left(1+o(1)\right)k^{-C\left(k\right)}.
\]
By utilizing the inclusion-exclusion principle, we conclude that 
\[
P(X)=P(\bigcup_{r}X_{r})\approx\sum_{r}P(X_{r})\approx\frac{\varphi\left(k\right)}{k^{C\left(k\right)}}
\]
whereas the union and the summation run through a complete residue
system $r$ modulo $k$. Hence, if $C\left(k\right)>1+\varepsilon$
for some fixed $\varepsilon>0$, we expect with a positive probability
that our $f\left(k\right)$ primes cover all invertible residue classes
at least once. On the other hand, if $C\left(k\right)<1-\varepsilon$
holds, we expect, by using the reversed Borel-Cantelli Lemma,\footnote{Cf. \cite{Chung1952}}
that $X$ is likely to occur infinitely often. Since $p_{n}\sim n\log n$,
the threshold $C=1$ amounts to the estimate $\beta\left(k\right)\approx\varphi\left(k\right)\log k\log\left(\varphi\left(k\right)\log k\right)=O\left(k\log^{2}k\right)$
for having about $\varphi\left(k\right)\log k$ primes in the interval
$\left[2,\beta(k)\right]$. This approximation was suggested by a
similar, but more complicated heuristic of Wagstaff \cite{Wagstaff1979}
and is plausible in view of various results e.g. from Turán \cite{Turan1936}.
The latter showed, assuming the Extended Riemann Hypothesis, that
for any $\delta>0$ the smallest prime $P\left(k,l\right)$ in the
invertible residue class $l$ modulo $k$ is exceeding the quantity
$\varphi\left(k\right)\log^{2+\delta}\left(k\right)$ for at most
$o\left(\varphi\left(k\right)\right)$ choices of $l$. There are
other results of this kind,\footnote{The latest record for calculating Linnik's constant, to the best of
our knowledge, is due to T. Xylouris \cite{Xylouris2011} who refined
the previous work of Heath-Brown \cite{Heath-Brown1992}. Moreover,
Heath-Brown conjectured that $P(k,l)=O(k\log^{2}k)$ holds for any
coprime pairs $l,k$ and $k$ sufficiently large.} we refer the reader to \cite{Granville1989} and the references therein.
However, there is also reason to be cautious with respect to the above
mentioned heuristic. In this direction there are, inter alia, the
results of Maier \cite{Maier1985}, Rubinstein and Sarnak \cite{Rubinstein1994},
or \cite{Knapowski1962}. 

Let us stress that for $k\in G$, where $G$ is as in Definition
\ref{def: Pinteger}, our heuristic suggests that one should need
about $\varphi\left(\left|k\right|\right)\log\left|k\right|$ primes
to cover the invertible residue classes modulo $\left|k\right|$ in
$G$ at least once with primes and not just $\varphi\left(\left|k\right|\right)+\omega\left(\left|k\right|\right)$
as one asks in Recaman's conjecture. Our first result proves, under
certain assumptions, that this is indeed the case. Furthermore, we
say $G$ satisfies Axiom $A$ (cf. \cite[p. 75]{Knopfmacher1990})
with $\delta>0$, if for some $0\leq\eta<\delta$ the counting function
$N_{G}\left(x\right):=\#\left\{ g\in G:\,\left|g\right|\leq x\right\} $
has the expansion $x^{\delta}+O(x^{\eta})$ as $x\rightarrow\infty$.
Thus we can state:
\begin{thm}
\label{thm:P integer finite in abstract setting}Let $K:=\left|k\right|$.
Let $G$ as in Definition \ref{def: Pinteger} satisfy Axiom $A$
with some $\delta>0$. Assume that numbers $\alpha=1$, $\beta\ll K\log^{a}K$
and $\iota\ll\log^{b}K$ are given for some fixed $a,b>0$ in the
case $0<\delta\leq1$ and in the case $\delta>1$ the value of $\beta$
may additionally differ from multiples of $K$ by at most $K^{1-\epsilon}$
for some absolute constant $\epsilon>0$. Then there are only finitely
many such $P^{*}$-integers. 
\end{thm}
For instance, the assumptions (on the semi-group) above are satisfied
if $G$ is the set of non-zero integral ideals of a number field $\mathfrak{K}$
with the usual ideal norm. Moreover, one can also interpret the
property to be a $P^{*}$-integer as the resolvability (in the set
of primes) of a certain set of Diophantine equations and inequalities.
For determining all such solutions, it is of interest to furnish Theorem
\ref{thm:P integer finite in abstract setting} with explicit bounds
on $k$ and it might be interesting in its own right to make a qualitative
statement quantitative. We shall do so only in the case $G=\mathbb{N}$
since one needs explicit bounds for the prime counting function $\pi_{G}\left(x\right):=\#\left\{ p\in G:\,g\,\mathrm{prime}\,,\left|g\right|\leq x\right\} $,
for $x>0$, of $G$ which are only known if one has sufficient arithmetic
information about $G$. For instance, the error term in Landau's prime
ideal theorem naturally depends on the given number field. However,
once these informations are given it is a straight forward task to
extend our explicit results to more general cases. 

Loosely speaking, our main result states, in a quantitative manner,
that blocks of primes (in the natural numbers) of approximate length
$\gamma\,\varphi\left(k\right)$ are, in general, not evenly distributed
among the reduced residue classes modulo $k$. More precisely, we
prove the following extension of Recaman's conjecture:
\begin{thm}
\textcolor{red}{\label{thm: main result general Pinterger finite}}Let
$\lambda\in\mathbb{N}\cup\left\{ 0\right\} $ and $d_{1},d_{2},d_{3}$
denote strictly positive real numbers. There are only finitely many
$P(\alpha,\beta,\gamma,\iota)$-integers in $\mathbb{N}$ such that
the growth restrictions $\alpha=\lambda k+O(k^{1-d_{1}})$, $\iota=O(k^{1-d_{2}})$
and $\beta=O(k\log^{d_{3}}k)$ are satisfied.
\end{thm}
The paper is organized as follows: Firstly, we deduce a necessary
condition for $g\in G$, where $G$ is always assumed to be as in
Definition \ref{def: Pinteger}, to be a $P^{*}$-integer and prove
Theorem \ref{thm:P integer finite in abstract setting}. This will
be done via a combinatorial argument which leads to inequalities involving
sums over the prime counting function $\pi$ evaluated at certain
points. Secondly, we will remove $\pi$ from these inequalities by
approximating it and then deal with the sums in such a manner that
we receive explicit formulas for seeing which large $k$ violate the
arising inequalities.

\section{Preliminaries and Proof of Theorem 1}

We first collect some results which we will need in the proofs. 
\begin{lem}[{Cf. \cite[Thm. 1]{trudgian2015updating}}]
\label{lem: estimate of chebyshev function vs. x}Let $\theta\left(x\right):=\sum_{p\leq x}\log p$
denote the Chebyshev function where the summation runs through all
primes $p\leq x$. With 
\begin{align*}
\varepsilon\left(x\right):= & \sqrt{\frac{8\log x}{17\pi\cdot\eta}}e^{-\sqrt{\eta^{-1}\log x}}\,\,\,\,\,\,\,\,\,\qquad\mbox{for }x\geq149,\eta:=6.455
\end{align*}
we have 
\[
\left|\theta\left(x\right)-x\right|<x\varepsilon\left(x\right),\qquad\mbox{for }x\geq149.
\]

\end{lem}

\begin{rem}
\label{rem: growth of nth prime and estiamte for prime conting function from below}We
recall that
\begin{enumerate}
\item $p_{n}\geq n\log n$ for any $n\geq1$, see \cite[p. 69]{J.B.Rosser1962},
and 
\item for $k\geq2\,953\,652\,287$ we have, cf. \cite[Thm. 6.9]{Dusart2010},
\begin{align*}
E_{0,-}:=2\pi(0.5k)-\pi(k) & >\frac{k}{\log(0.5k)}\left(1+\frac{1}{\log(0.5k)}+\frac{2}{\log^{2}(0.5k)}\right)\\
 & -\frac{k}{\log(k)}\left(1+\frac{1}{\log(k)}+\frac{2.334}{\log^{2}(k)}\right).
\end{align*}

\item Moreover, we need the estimates 
\begin{equation}
\sqrt{\frac{2}{\pi}\left(2S+1\right)}\leq\prod_{s=1}^{S}\frac{2s+1}{2s}\leq\frac{2S+1}{\sqrt{S\pi}}\label{eq: wallis estimates}
\end{equation}
which are well-known (in equivalent forms) in the context of Wallis'
product formula for $\pi$, cf. \cite[p. 504-505]{Heuser2009}.
\item The following estimate holds, cf. \cite[p. 72]{J.B.Rosser1962}: 
\[
\varphi\left(k\right)\geq\frac{k}{1.7811\log\log k+\frac{2.51}{\log\log k}},\qquad k\geq3.
\]

\item Let $\li\left(x\right)$ denote the integral $\int_{2}^{x}\frac{\mathrm{d}\tau}{\log\tau}$
for $x>0$. If for an arithmetical semi-group $G$ the counting functions
$g\left(x\right):=\#\left\{ g\in G:\,\,\left|g\right|\leq x\right\} $
takes the form 
\[
g\left(x\right)=Ax^{\delta}+O(x^{\delta}\log^{-\beta}x),\qquad\beta>3,\;\delta>0,\;x\rightarrow\infty,
\]
 then the prime counting function of $G$ can be written as 
\[
\pi_{G}\left(x\right)=\mathrm{li}(x^{\delta})+O(x^{\delta}\log^{-c}x)\qquad\mbox{for any }c<\frac{\beta}{3}.
\]
This is due to Wegmann \cite{Wegmann1966}. In particular, the conclusion
is true, if $G$ satisfies Axiom $A$. 
\end{enumerate}
\end{rem}
Our method to detect $P^{*}$-integers originates from \cite{L.Hajdu2012},
which we shall describe in the following. We write $\pi_{G}\left(x\right)=\pi\left(x\right)$
and denote for natural numbers $x,K$ by $x\mod K$ the unique remainder
$r\in\left\{ 0,\ldots,K-1\right\} $ such that $x=qK+r$ holds for
some $q\in\mathbb{N}$. Let us assume that $k$ is a $P^{*}$-integer
and put $K:=\left|k\right|$. Then, by the symmetry of coprime residue
classes modulo $K$ about  $0.5K$, the cardinalities of the sets
\begin{align*}
A_{1} & :=\left\{ p\in G:\:\alpha\leq\left|p\right|\leq\beta,\,p\,\mathrm{prime},\,\left|p\right|\mod K\leq0.5K\right\} ,\\
A_{2} & :=\left\{ p\in G:\:\alpha\leq\left|p\right|\leq\beta,\,p\,\mathrm{prime},\,\left|p\right|\mod K>0.5K\right\} ,
\end{align*}
differ by at most $\iota$ elements. For checking this condition,
we count the size of $A_{i}$ which is done by the following lemma:
\begin{lem}
\label{lem: counting the sets A and B}Let 
\[
E_{j,1}\left(k\right):=\pi((j+0.5)K)-\pi(jK-1),\qquad E_{j,2}\left(k\right):=\pi((j+1)K)-\pi((j+0.5)K)
\]
for $j\geq0,\,i=1,2$. If $\lambda,\Lambda$ denote integers such
that $\lambda K\leq\alpha<\left(\lambda+1\right)K$, and $\Lambda K\leq\beta<\left(\Lambda+1\right)K$
hold, then we have 
\begin{equation}
\bigl|A_{i}\bigr|=M_{i}\left(k\right)+\sum_{j\in\mathcal{I}}E_{j,i}\left(k\right),\qquad\mathcal{I}:=\mathcal{I}_{\lambda,\Lambda}:=\left\{ \lambda+1,\lambda+2\ldots,\Lambda-1\right\} ,\label{eq: size of Ai}
\end{equation}
whereas $M_{i}\left(k\right)$ is defined in (\ref{eq: definition delta1}). \end{lem}
\begin{proof}
We partition the set $A_{1}$ into subsets $A_{1,j}$ of primes having
norm in

$\left[jK,\left(j+0.5\right)K\right]$ and $A_{2}$ into subsets of
primes having norm in $\left[\left(j+0.5\right)K,(j+1)K\right]$ where
$\lambda\leq j\leq\Lambda$. Note that $E_{j,i}\left(k\right)$ counts
how many primes are located in $A_{i,j}$ for $\lambda<j<\Lambda$
and $i=1,2$. This gives rise to the term $\sum_{j\in\mathcal{I}}E_{j,i}\left(k\right)$.
Counting the primes near the end-points $j=\lambda$ and $\Lambda$
demands more care because one needs to distinguish whether $\alpha-\lambda K\leq0.5K$
holds or not and whether $\beta-\Lambda K\leq0.5K$ holds or not in
order to start or stop counting with the suitable $A_{i,\lambda}$
or $A_{i,\Lambda}$. Thus, we get four cases to which we shall refer
to in the following manner:
\begin{table}[H]
\selectlanguage{english}%
\caption{\label{tab:table-label cases}}
\foreignlanguage{american}{}%
\begin{tabular}{|c|c|c|}
\hline 
\selectlanguage{american}%
condition\selectlanguage{english}%
 & \selectlanguage{american}%
$\alpha-\lambda K\leq0.5K$\selectlanguage{english}%
 & \selectlanguage{american}%
$\alpha-\lambda K>0.5K$\selectlanguage{english}%
\tabularnewline
\hline 
\hline 
\selectlanguage{american}%
$\beta-\Lambda K\leq0.5K$\selectlanguage{english}%
 & \selectlanguage{american}%
case (i)\selectlanguage{english}%
 & \selectlanguage{american}%
case (iii)\selectlanguage{english}%
\tabularnewline
\hline 
\selectlanguage{american}%
$\beta-\Lambda K>0.5K$\selectlanguage{english}%
 & \selectlanguage{american}%
case (ii)\selectlanguage{english}%
 & \selectlanguage{american}%
case (iv)\selectlanguage{english}%
\tabularnewline
\hline 
\end{tabular}\selectlanguage{american}%
\end{table}
In view of equation (\ref{eq: size of Ai}), we can define the proclaimed
functions $M_{i}$ by using (henceforth) the short hand notation $x_{j}:=jK$,
$\overline{x}_{j}:=\frac{x_{j}+x_{j+1}}{2}$ via 
\begin{align}
M_{1}\left(k\right) & :=\begin{cases}
\pi(\overline{x}_{\lambda})-\pi\left(\alpha-1\right)+\pi(\beta)-\pi(x_{\Lambda}) & \mbox{ in case (i)}\\
\pi(\overline{x}_{\lambda})-\pi\left(\alpha-1\right)+E_{\Lambda,1}\left(k\right) & \mbox{ in case (ii)}\\
\pi\left(\beta\right)-\pi(x_{\Lambda}-1)\,\,\,\,\,\,\,\,\,\,\,\,\,\,\,\,\,\,\,\,\,\,\,\,\,\,\,\,\,\,\,\,\,\,\,\,\,\,\,\,\,\,\,\,\,\,\,\, & \mbox{ in case (iii)}\\
E_{\Lambda,1}\left(k\right) & \mbox{ in case (iv)}
\end{cases},\label{eq: definition delta1}\\
M_{2}\left(k\right) & :=\begin{cases}
E_{\lambda,2}\left(k\right) & \mbox{ in case (i)}\\
E_{\lambda,2}\left(k\right)+\pi(\beta)-\pi(\overline{x}_{\Lambda}) & \mbox{ in case (ii)}\\
\pi(x_{\lambda+1})-\pi\left(\alpha-1\right) & \mbox{ in case (iii)}\\
\pi(x_{\lambda+1})-\pi\left(\alpha-1\right)+\pi\left(\beta\right)-\pi(\overline{x}_{\Lambda}) & \mbox{ in case (iv)}
\end{cases}.\qedhere\label{eq: definition delta2}
\end{align}
 
\end{proof}
It is useful to put $E_{j}\left(k\right):=E_{j,1}\left(k\right)-E_{j,2}\left(k\right)$,
$M\left(k\right):=M_{1}\left(k\right)-M_{2}\left(k\right)$, for writing
\begin{equation}
\bigl|A_{1}\bigr|-\bigl|A_{2}\bigr|=M\left(k\right)+\sum_{j\in\mathcal{I}}E_{j}\left(k\right).\label{eq: difference of the Ai}
\end{equation}
Moreover, we say an assertion $A\left(k\right)$ concerning natural
numbers is eventually true if there exists a $k_{0}\in\mathbb{N}$
such that $A\left(k\right)$ holds true for all $k\geq k_{0}$.
\begin{proof}[Proof of Theorem \ref{thm:P integer finite in abstract setting}]
Since $\alpha=1$ we may assume $\lambda=0$, and that either case
(i) or (ii) of Table \ref{tab:table-label cases} occurs. Let $0<\delta\leq1$
for the moment. Remark \ref{rem: growth of nth prime and estiamte for prime conting function from below}
gives an approximation for the prime counting function from which
we infer 
\[
M\left(k\right)\geq2\li((0.5K)^{\delta})-\li(K^{\delta})+E_{\Lambda}+O(K^{\delta}\log^{-\eta}(0.5K)),\qquad\eta>0.
\]
Moreover, we have 
\begin{equation}
2\li((0.5K)^{\delta})-\li(K^{\delta})=\int_{2}^{K^{\delta}}\frac{2^{1-\delta}-1+\frac{\delta\log2}{\log(\tau)}}{\log(2^{-\delta}\tau)}\mathrm{d}\tau,\qquad\delta>0.\label{eq: difference li in proof of abstract p integer}
\end{equation}
Since the derivative of $x\mapsto\li(x^{\delta})$ is eventually decreasing,
it follows from the mean value theorem that $2\li(\overline{x}_{j}^{\delta})-\li(x_{j}^{\delta})-\li(x_{j+1}^{\delta})$
is eventually positive for any $j\geq1$. Hence, we conclude that
\begin{equation}
\sum_{j=1}^{\Lambda}E_{j}\left(k\right)>\left(\Lambda-1\right)O(K^{\delta}\log^{-\eta}(0.5K)),\qquad\eta>0.\label{eq: sum over Ej}
\end{equation}
Using Equation (\ref{eq: difference of the Ai}) and the above estimate
we find that 
\[
\bigl|A_{1}\bigr|-\bigl|A_{2}\bigr|>\frac{K^{\delta}\log2}{\log(K^{\delta})\log(0.5K)}+\left(\Lambda-1\right)O(K^{\delta}\log^{-\eta}(0.5K)),
\]
which proves the claim in the case $0<\delta\leq1$. Now let $\delta>1$.
Then the difference $2\li(\overline{x}_{j}^{\delta})-\li(x_{j}^{\delta})-\li(x_{j+1}^{\delta})$
is negative for any $j\geq1$. We note that $M(k)$ is bounded from
above by $2\li((0.5K)^{\delta})-\li(K^{\delta})$ up to an error term
 
\[
O(K^{\delta}\log^{-\eta}(0.5K))+\begin{cases}
\li(\beta^{\delta})-\li((x_{\Lambda}-1)^{\delta}) & \mbox{ in case (i)}\\
\li(x_{\Lambda+1}^{\delta})-\li(\beta^{\delta}) & \mbox{ in case (ii)}
\end{cases}.
\]
The assumption on $\beta$ implies that the expressions in the brackets
are in $O(K^{\delta-\epsilon})$ for some $\varepsilon>0$ and hence
$O(K^{\delta}\log^{-\eta}(0.5K))$. Therefore, we obtain from (\ref{eq: difference li in proof of abstract p integer})
that for some suitable constant $c>0$ the estimate 
\[
M\left(k\right)<\frac{-cK^{\delta}}{\delta\log(K)}+O(K^{\delta}\log^{-\eta}(0.5K))
\]
holds. Because the left hand side of (\ref{eq: sum over Ej}) is bounded
by $\left(\Lambda-1\right)O(K^{\delta}\log^{-\eta}(0.5K))$, we conclude
from (\ref{eq: difference of the Ai}) that $-\iota<\left|A_{1}\right|-\left|A_{2}\right|$
is eventually violated.
\end{proof}

\section{Auxiliary Results}

In what follows we investigate conditions for a natural number $k$
to be a $P^{*}$-integer. It is important to notice, that $M$ is
strictly positive in case (i) and (can be) strictly negative in case
(iv) of table (\ref{tab:table-label cases}). Therefore, upper\emph{
and }lower bounds are needed, in order to derive the asymptotic of
the difference in (\ref{eq: difference of the Ai}). In order to prove
Theorem \ref{thm: main result general Pinterger finite}, it suffices
to derive lower a bound, though upper bounds can be derived in the
same way. This is done by the following two results.
\begin{lem}
\label{lem: estimate for Ejik}Let $k\geq2\,953\,652\,287$, $\varepsilon$
as in Lemma \ref{lem: estimate of chebyshev function vs. x}, $x_{j}=kj$,
and $j$ be a natural number. Define the functions 
\begin{align*}
E_{j,-}\left(k\right) & :=2\overline{x}_{j}\frac{1-\varepsilon(\overline{x}_{j})}{\log\overline{x}_{j}}-x_{j}\frac{1+\varepsilon(x_{j})}{\log x_{j}}-x_{j+1}\frac{1+\varepsilon(x_{j+1})}{\log x_{j+1}},
\end{align*}
and
\[
r_{j}\left(k\right):=\frac{k\varepsilon(\overline{x}_{j})}{\log^{2}\overline{x}_{j}},\qquad r_{0}\left(k\right):=0.
\]
Then the inequality 
\begin{equation}
E_{j,-}\left(k\right)-r_{j}\left(k\right)<E_{j}\left(k\right)\label{eq: estimate for growth of Ej}
\end{equation}
holds for $j\geq0$.\end{lem}
\begin{proof}
We apply the well-known formula 
\begin{equation}
\pi\left(x\right)=\frac{\theta\left(x\right)}{\log\left(x\right)}+\int_{2}^{x}\frac{\theta\left(\tau\right)}{\tau\log^{2}\tau}\,\mbox{d}\tau\label{eq: relation Pi function and Chebyshev function}
\end{equation}
to see that $E_{j}\left(k\right)$ equals the sum 
\begin{align*}
 & \frac{2\theta(\overline{x}_{j})}{\log\overline{x}_{j}}-\frac{\theta(x_{j})}{\log x_{j}}-\frac{\theta(x_{j+1})}{\log x_{j+1}}+\int_{x_{j}}^{\overline{x}_{j}}\frac{\theta\left(\tau\right)}{\tau\log^{2}\tau}\,\mbox{d}\tau-\int_{\overline{x}_{j}}^{x_{j+1}}\frac{\theta\left(\tau\right)}{\tau\log^{2}\tau}\,\mbox{d}\tau.
\end{align*}
Lemma \ref{lem: estimate of chebyshev function vs. x} for $j\geq1$
and Remark \ref{rem: growth of nth prime and estiamte for prime conting function from below}
for $j=0$ yield that the first three terms above exceed $E_{j,-}\left(k\right)$
for $j\geq0$. By using Lemma \ref{lem: estimate of chebyshev function vs. x},
we infer 
\[
\int_{x_{j}}^{\overline{x}_{j}}\frac{\theta\left(\tau\right)}{\tau\log^{2}\tau}\,\mbox{d}\tau-\int_{\overline{x}_{j}}^{x_{j+1}}\frac{\theta\left(\tau\right)}{\tau\log^{2}\tau}\,\mbox{d}\tau>\frac{k}{2}\frac{1-\varepsilon(\overline{x}_{j})}{\log^{2}\overline{x}_{j}}-\frac{k}{2}\frac{1+\varepsilon(\overline{x}_{j})}{\log^{2}\overline{x}_{j}}=r_{j}\left(k\right)
\]
which implies (\ref{eq: estimate for growth of Ej}).
\end{proof}
Observing that 
\begin{align*}
M\left(k\right)= & \begin{cases}
E_{\lambda}(k)+\pi(x_{\lambda})-\pi(\alpha-1)+\pi(\beta)-\pi(x_{\Lambda}) & \mbox{ in case (i)}\\
E_{\lambda}(k)+\pi(x_{\lambda})-\pi(\alpha-1)+2\pi(\overline{x}_{\Lambda})-\pi(x_{\Lambda})-\pi(\beta)\,\,\,\,\,\,\,\,\,\,\,\,\, & \mbox{ in case (ii)}
\end{cases}
\end{align*}
we derive the following technical but crucial corollary.
\begin{cor}
\label{cor: estimate from below for the growth main term}The term
$M(k)$ is bounded from below in the cases $(i)-(ii)$ by $E_{\lambda,-}\left(k\right)-r_{\lambda}\left(k\right)-\Delta\left(\lambda,k\right)+R\left(k\right)$
whereas we put 
\[
\Delta\left(\lambda,k\right):=-\begin{cases}
\pi(\alpha-1) & \mbox{ if }\lambda=0\\
\frac{\alpha}{\log x_{\lambda}}\bigl(1+\tilde{\Delta}(x_{\lambda},\alpha)\bigr) & \mbox{ if }\lambda>0
\end{cases},
\]
$R\left(k\right):=0$ in case $(i)$ and $R\left(k\right):=E_{\Lambda,-}\left(k\right)-r_{\Lambda}\left(k\right)$
in case $(ii)$ and define 
\[
\tilde{\Delta}(x_{-},x_{+}):=\Bigl(1-\frac{x_{-}}{x_{+}}\Bigr)\frac{1+\varepsilon(x_{-})}{\log^{2}x_{-}}-\frac{x_{-}}{x_{+}}+2\varepsilon(x_{-}),\qquad0<x_{-}\leq x_{+}.
\]
\end{cor}
\begin{proof}
The inequality 
\begin{align}
\pi(x_{+})-\pi(x_{-}) & <\frac{x_{+}}{\log x_{-}}\bigl(1+\tilde{\Delta}(x_{-},x_{+})\bigr)\label{eq: difference pi plus pi minus}
\end{align}
can be deduced from Equation (\ref{eq: relation Pi function and Chebyshev function})
via 
\begin{align*}
\pi(x_{+})-\pi(x_{-}) & <x_{+}\frac{1+\varepsilon(x_{+})}{\log x_{+}}-x_{-}\frac{1-\varepsilon(x_{-})}{\log x_{-}}+\int_{x_{-}}^{x_{+}}\frac{1+\varepsilon\left(t\right)}{\log^{2}t}\,\mathrm{d}t\\
 & <\frac{x_{+}}{\log x_{+}}-\frac{x_{-}}{\log x_{-}}+2\frac{x_{+}\varepsilon(x_{-})}{\log x_{-}}+(x_{+}-x_{-})\frac{1+\varepsilon(x_{-})}{\log^{2}x_{-}}
\end{align*}
and bracketing out the term $\frac{x_{+}}{\log x_{-}}$ on the right
hand side. Let $\lambda\geq1$. Using the Estimate (\ref{eq: difference pi plus pi minus})
with $x_{+}:=\alpha$ and $x_{-}:=x_{\lambda}$, we get 
\begin{equation}
\pi(\alpha-1)-\pi(x_{\lambda})<\frac{\alpha}{\log x_{\lambda}}\bigl(1+\Delta(x_{\lambda},\alpha)\bigr).\label{eq: estimate t minus pi of xlambda plus one}
\end{equation}
In the cases (i), (ii) the claim follows now by 
\begin{equation}
E_{\lambda,-}\left(k\right)+\pi(x_{\lambda})=2\pi(\overline{x}_{\lambda})-\pi(x_{\lambda+1}),\qquad E_{\Lambda,-}\left(k\right)<2\pi(\overline{x}_{\Lambda})-\pi(x_{\Lambda})-\pi(\beta),\label{eq: proof of growth of main term proof cases one and two errorterm e}
\end{equation}
and applying Lemma \ref{lem: estimate for Ejik}. If $\lambda=0$,
then the claim follows in the cases (i), and (ii) directly from the
estimate (\ref{eq: proof of growth of main term proof cases one and two errorterm e})
and Remark \ref{rem: growth of nth prime and estiamte for prime conting function from below}.
\end{proof}
Since we know explicit bounds for the growth of the term $M$, we
need to derive explicit bounds for 
\[
\sum_{j\in\mathcal{I}}E_{j}.
\]
In view of Lemma \ref{lem: estimate for Ejik}, we can concentrate
on dealing with sums 
\begin{equation}
\sum_{j=a}^{b}E_{j,-}\left(k\right).\label{eq: sums over ejminus general type}
\end{equation}
To this end, we define $f\left(x\right):=x\left(\log x\right)^{-1}$,
and note that $E_{j,-}\left(k\right)$ splits into 
\begin{align*}
 & 2f(\overline{x}_{j})-f(x_{j})-f(x_{j+1})-2\varepsilon(\overline{x}_{j})f(\overline{x}_{j})-\varepsilon(x_{j})f(x_{j})-\varepsilon(x_{j+1})f(x_{j+1}).
\end{align*}
Let $E_{j}'\left(k\right)$ denote the first three terms above, and
let $E_{j}''\left(k\right)$ denote the remaining three. For deriving
explicit lower and upper bounds for sums over $E_{j,-}\left(k\right)$,
it suffices to deal with the (slightly easier) sums over $E_{j}'\left(k\right)$
and $E_{j}''\left(k\right)$. This will be done in the following.
\begin{lem}
\label{lem: growth of Eijk prime}For natural numbers $a\leq b$ and
$k\geq e^{4}$ we have the following estimate 
\begin{equation}
\frac{8}{k}\sum_{j=a}^{b}E_{j}'\left(k\right)>\frac{\log\frac{4b+6}{9a}}{\log^{2}(x_{b+1})}.\label{eq: lower bound for sum over Ejk prime}
\end{equation}
\end{lem}
\begin{proof}
Let us note that 
\begin{equation}
E_{j}'\left(k\right)=\int_{x_{j}}^{\overline{x}_{j}}f'\left(x\right)-f'\left(x+0.5k\right)\,\mbox{d}x.\label{eq: writing Ejprime as integral over derivative}
\end{equation}
Observing that $f'(x)-f'(x+0.5k)$ equals
\[
\biggl(\frac{1}{\log x}-\frac{1}{\log\left(x+0.5k\right)}\biggr)\biggl(1-\biggl(\frac{1}{\log x}+\frac{1}{\log\left(x+0.5k\right)}\biggr)\biggr)
\]
we infer, since $k\geq e^{4}$, the inequality 
\[
\frac{1}{2}\frac{\log(1+\frac{k}{2x})}{\log(x)\log(x+0.5k)}<f'(x)-f'(x+0.5k),\qquad x\in[x_{j},\overline{x}_{j}],\quad j\geq1.
\]
Integrating with respect to $x$ from $x_{j}$ to $\overline{x}_{j}$,
in view of (\ref{eq: writing Ejprime as integral over derivative}),
and summing over $j$ yields 
\[
\frac{k}{4}\sum_{j=a}^{b}\frac{\log(1+\frac{1}{2j+2})}{\log(\overline{x}_{j})\log(x_{j+1})}<\sum_{j=a}^{b}E_{j}'\left(k\right).
\]
By using partial summation, we obtain 
\begin{equation}
\sum_{j=a}^{b}\frac{\log(1+\frac{1}{2j+2})}{\log^{2}(x_{j+1})}>\frac{\log\prod_{s=a}^{b}\frac{2(s+1)+1}{2(s+1)}}{\log^{2}(x_{b+1})}.\label{eq: partial summation for Eijk'}
\end{equation}
The estimates (\ref{eq: wallis estimates}) imply that the product
in the numerator above can be bounded from below by $(4b+6)^{0.5}(9a)^{-0.5}$.
Therefore, we obtain (\ref{eq: lower bound for sum over Ejk prime})
from (\ref{eq: partial summation for Eijk'}).
\end{proof}
With the above estimates at hand, we can derive lower bounds on (\ref{eq: sums over ejminus general type}).
\begin{cor}
\label{cor: estimates for the growth the sum of wj plus or minus}Let
$j\geq1$, $a\leq b$ denote natural numbers and\textup{ $\sigma_{a,b}:=\sum_{j=a}^{b}j$}.
Then 
\begin{align*}
\sum_{j=a}^{b}\frac{E_{j,-}\left(k\right)-r_{j}\left(k\right)}{k}>\frac{\log\frac{4b+6}{9a}}{8\log^{2}(x_{b+1})}-5\frac{\varepsilon(x_{a})}{\log(x_{a})}\sigma_{a+1,b+1}
\end{align*}
holds.\end{cor}
\begin{proof}
Let us note that 
\begin{align*}
\sum_{j=a}^{b}\frac{\varepsilon(x_{j})j}{\log(x_{j})} & <\frac{\varepsilon(x_{a})}{\log(x_{a})}\sigma_{a,b}\quad\mbox{ and }\quad\sum_{j=a}^{b}\frac{\varepsilon(\overline{x}_{j})\left(j+0.5\right)}{\log(\overline{x}_{j})}<\frac{\varepsilon(x_{a})}{\log(x_{a})}\sigma_{a+1,b+1}
\end{align*}
hold. Observing $\sigma_{a+1,b+1}\geq\sigma_{a,b}$ implies 
\begin{align}
 & \frac{1}{k}\sum_{j=a}^{b}E_{j}''\left(k\right)<4\frac{\varepsilon(x_{a})}{\log(x_{a})}\sigma_{a+1,b+1}.\label{eq: growth of sum over Ej prime prime-1}
\end{align}
By using (\ref{eq: lower bound for sum over Ejk prime}) and (\ref{eq: growth of sum over Ej prime prime-1}),
we deduce 
\begin{align*}
\frac{1}{k}\sum_{j=a}^{b}(E_{j}'\left(k\right)-E_{j}''\left(k\right)) & >\frac{\log\frac{4b+6}{9a}}{8\log^{2}(x_{b+1})}-4\frac{\varepsilon(x_{a})}{\log(x_{a})}\sigma_{a+1,b+1}.
\end{align*}
Combining this inequality with the obvious upper bounds for $\frac{1}{k}\sum_{j=a}^{b}r_{j}\left(k\right)$
while using $\sigma_{a+1,b+1}\geq(b-a+1)$ yields the claim. 
\end{proof}

\section{Proof of the Main Theorem }

\begin{proof}[Proof of Theorem \ref{thm: main result general Pinterger finite}]
It suffices to establish that 
\[
S\left(k\right):=\bigl|A_{1}\bigr|-\bigl|A_{2}\bigr|-\iota
\]
is eventually strictly positive. Assume for the moment that we are
in the cases (i) or (ii) of Table \ref{tab:table-label cases}. Equation
(\ref{eq: difference of the Ai}) and Lemma \ref{lem: estimate for Ejik}
imply 
\[
S\left(k\right)>M\left(k\right)-\iota+\sum_{j\in\mathcal{I}}(E_{j,-}\left(k\right)-r_{j}\left(k\right)).
\]
By using Corollary \ref{cor: estimate from below for the growth main term},
we deduce that $S\left(k\right)$ exceeds 
\begin{equation}
R(k)-\Delta\left(\lambda,k\right)-\iota+\sum_{j=\lambda}^{\Lambda-1}(E_{j,-}\left(k\right)-r_{j}\left(k\right)).\label{eq: S of k should exceed this quantity}
\end{equation}
Let $\lambda\ge1$ and define $b=\Lambda-1$ in case $(i)$ and $b=\Lambda$
in case $(ii)$. Then applying Corollary \ref{cor: estimates for the growth the sum of wj plus or minus}
with $a=\lambda,\,b$ yields that it suffices to check whether 
\begin{equation}
-\frac{\alpha}{k}\frac{1+\tilde{\Delta}(x_{\lambda},\alpha)}{\log x_{\lambda}}-\frac{\iota}{k}+\frac{\log\frac{4b+6}{9\lambda}}{8\log^{2}(x_{b+1})}-5\frac{\varepsilon(k)}{\log(k)}\sigma_{1,b+1}>0.\label{eq: proof of the main theorem equation in first case}
\end{equation}
As $x_{\lambda}\alpha^{-1}-1<Ck^{-d_{1}}$ holds for some $C>0$,
there is an explicitly computable $C_{1}>0$ such that $1+\tilde{\Delta}(x_{\lambda},\alpha)<C_{1}\varepsilon(k)$.
Hence, we can estimate the left hand side of (\ref{eq: proof of the main theorem equation in first case})
from below by 
\[
-C\varepsilon(k)-\frac{\iota}{k}+\frac{\log\frac{4b+6}{9\lambda}}{8\log^{2}(x_{b+1})}-\frac{5\varepsilon(k)}{\log(k)}\sigma_{1,b+1}.
\]
Using the bounds $b+2\leq C_{3}\log^{d_{3}}k$, $\iota<C_{2}k^{1-d_{2}}$
with some $C_{2},C_{3}>0$ yields that it suffices to prove that 
\begin{equation}
\frac{\log\frac{4b+6}{9\lambda}}{8\log^{2}(x_{b+1})}-\frac{5}{4}\varepsilon(k)C_{3}^{2}\log^{2d_{3}-1}(k)-C_{1}\varepsilon(k)-\frac{C_{2}}{k^{d_{2}}}>0\label{eq: fininal inequality in the case lambda greater or equal 1}
\end{equation}
is positive. This is certainly true for sufficiently large $k$ if
we can establish that $\frac{4b+6}{9\lambda}$ exceeds $1$ eventually.
Since for a $P^{*}$-integer $\gamma\geq1$ implies $\beta\geq p_{\varphi\left(k\right)}$,
we conclude from Remark \ref{rem: growth of nth prime and estiamte for prime conting function from below}
that 
\[
\beta>\varphi\left(k\right)\log\varphi(k)\gg k\frac{\log k}{\log\log k}.
\]
Hence, $b$ can be assumed to be arbitrarily large, as desired. Now
let $\lambda=0$. Applying Corollary \ref{cor: estimates for the growth the sum of wj plus or minus}
with $a=1$, and $b$ as before, we deduce from (\ref{eq: S of k should exceed this quantity})
that it suffices to check whether 
\[
E_{0,-}\left(k\right)-\frac{\pi(\alpha)+\iota}{k}+\frac{\log\frac{4b+6}{9}}{8\log^{2}(x_{b+1})}-5\frac{\varepsilon(k)}{\log(k)}\sigma_{1,b+1}>0.
\]
Since $\pi(\alpha)<C_{1}k^{1-d_{1}}$, $\iota<C_{2}k^{1-d_{2}}$ and
$\sigma_{1,b+1}\leq C_{3}^{2}\log^{2d_{3}-1}k$ we see that we need
to check
\begin{equation}
E_{0,-}\left(k\right)+\frac{\log\frac{4b+6}{9}}{8\log^{2}(x_{b+1})}-5\varepsilon(k)C_{3}^{2}\log^{2d_{3}-1}k-C_{1}k^{-d_{1}}-C_{2}k^{-d_{2}}>0,\label{eq: final inequlatity in the case lambda equals zero}
\end{equation}
which is satisfied for sufficiently large $k$. This proves the claim
in the cases (i) or (ii). In the case (iii) or (iv), we write $\alpha=x_{\lambda}-\Delta$
for some $0<\Delta=O(k^{1-d_{1}})$. In comparison to $S(k)$ in the
cases (i) and (ii), we have to add the additional expression $E=\pi(x_{\lambda}+\Delta)-\pi(x_{\lambda})-(\pi(x_{\lambda})-\pi(x_{\lambda}-\Delta))$
to the former $S(k)$. One checks easily that $E=O(x_{\lambda}\varepsilon(x_{\lambda}))$.
Hence, $E$ can not effect the sign of $S(k)$ for large $k$ in the
cases (i) and (ii) since it's order is lower than the order of $S(k)$,
as we see by considering the terms in (\ref{eq: fininal inequality in the case lambda greater or equal 1})
and (\ref{eq: final inequlatity in the case lambda equals zero}).
This completes the proof.
\end{proof}
Using the above proof we can state explicit bounds on certain kinds
of $P^{*}$-integers. 
\begin{cor}
Let $b+2\leq C_{3}\log^{d_{3}}k$, $\iota<C_{2}k^{1-d_{2}}$ with
some $C_{2},C_{3}>0$. Under the assumptions of Theorem \ref{thm: main result general Pinterger finite}
there is an effectively computable number $C_{0}>0$ such that every
natural number $k\geq C_{0}$ satisfying (\ref{eq: fininal inequality in the case lambda greater or equal 1})
if $\lambda\geq1$, or (\ref{eq: final inequlatity in the case lambda equals zero})
if $\lambda=0$ is not such a $P(\alpha,\beta,\gamma,\iota)$-integer.\end{cor}
\begin{rem}
Let us add some further comments:\end{rem}
\begin{itemize}
\item It poses no general problem to modify our arguments to study the distribution
of other sequences in residue classes, since we essentially employed
the euclidean structure, properties of the norm function, and the
growth properties of the prime counting function. E.g. one can derive
similar results about the distribution of numbers or elements with
$s$-prime factors where $s$ is a fixed natural number, while considering
semi-groups with the just mentioned properties. 
\item Moreover, one could slightly relax the growth restriction in Theorem
\ref{thm: main result general Pinterger finite} and still conclude
finiteness of such $P^{*}$-integers. However, this would only complicate
the technical aspects of the proof and bring no deeper insight.
\end{itemize}
\newpage{}
\bibliographystyle{plaindin}

\end{document}